\documentclass[11pt]{amsproc}
\usepackage{amsfonts}

\usepackage{mathrsfs}
\usepackage{pifont}
\usepackage{amsmath}
\usepackage{amssymb}
\usepackage{graphics}
\newtheorem{theorem}{Theorem}[section]
\newtheorem{lemma}[theorem]{Lemma}

\newtheorem{proposition}[theorem]{Proposition}

\theoremstyle{definition}
\newtheorem{definition}[theorem]{Definition}

\def\al{\alpha}

\numberwithin{equation}{section}

\begin{document}

\title[Vertex operators for symplectic/orthogonal Schur functions]
{Vertex operators, Weyl determinant formulae and Littlewood duality}
\author{Naihuan Jing and Benzhi Nie}
\address{Department of Mathematics, North Carolina State University, Raleigh, NC 27695, USA
and School of Mathematical Sciences, South China University of Technology,
Guangzhou, Guangdong 510640, China}
\email{jing@math.ncsu.edu}
\address{School of Mathematical Sciences, South China University of Technology,
Guangzhou, Guangdong 510640, China}
\email{niebenzhi@163.com}

\thanks{{\scriptsize
\hskip -0.4 true cm MSC (2010): Primary: 17B69; Secondary: 05E05.
\newline Keywords: Weyl character formulae,
 vertex operators, symplectic/orthogonal Schur Functions
}}

\maketitle

\begin{abstract}
Vertex operator realizations of symplectic and orthogonal Schur functions are studied and
expanded. New proofs of determinant identities of irreducible characters for
the symplectic and orthogonal groups
are given. We also give a
new proof of the duality between the universal orthogonal and symplectic Schur functions
using vertex operators.
\end{abstract}

\section{Introduction}

Symmetric functions (cf. \cite{Ma}) were used by Hermann Weyl to determine
irreducible characters of highest weight representations of the classical groups \cite{W}
as consequence of the Cauchy identities.
Later Dudley Littlewood \cite{Li} upgraded the approach of symmetric functions and studied
the symplectic and orthogonal Schur functions in the same setting
as the Schur symmetric functions, and showed that in respective
categories $\mathcal O$, the restriction functor $Res^{\mbox{GL}_{2n}}_{\mbox{Sp}_{2n}}$ has the same
decomposition (up to duality)
as that of $Res^{\mbox{GL}_{2n}}_{\mbox{SO}_{2n}}$ or $Res^{\mbox{GL}_{2n+1}}_{\mbox{SO}_{2n+1}}$,
which in turn would imply that symplectic Schur functions are equal to orthogonal Schur functions
with the conjugate Young diagrams.

Vertex operator realization of Schur symmetric functions, though relatively young,
was started in the early days of its appearance in
representations of affine Lie algebras (cf. \cite{Z}, see also \cite{J1}).
Historically
the Kyoto school's fermionic formulation (cf \cite{DJKM} for a beautiful survey) had also used Schur functions
in a remarkable way to understand the KP and KdV hiearachies earlier than the bosonic consideration.
The aim of this paper is to understand both the Weyl
determinant formulae and the Littlewood duality from the vertex operator
viewpoint, and also explain why
many properties of Schur functions are shared by symplectic and orthogonal Schur functions.

In the classic paper \cite{KT} the symplectic and orthogonal Schur functions have been systematically
studied by Koike and Terada following Weyl \cite{W}, where they obtained
determinant formulae in terms of elementary symmetric functions (see also \cite{K}).
Parallel to their approach we will
define certain vertex
operators to realize the Schur symplectic and orthogonal symmetric functions
inside the ring $\Lambda$ of symmetric functions. The vertex operators we constructed are closely related to
Baker's vertex operators \cite{B} who used certain middle terms (similar to \cite{J1}), but they
are not necessary for our purpose (see \cite{J1, J3}).
We then compute directly that the symplectic and orthogonal Schur functions can be realized
by vertex operators within $\Lambda$. Due to Clifford type relations satisfied by
two basis generators we show that
they are actually the same up to a sign, from this we then obtain a new proof for Littlewood's
duality between symplectic and orthogonal
Schur functions. Finally eight determinant formulae are easily derived for both
symplectic and orthogonal Schur functions (see Theorem \ref{T:fourdet}).

We remark that some of the vertex operators were studied by Shimozono and Zabrocki
 \cite{SZ} who constructed symplectic and orthogonal Schur functions
in the language of $\lambda$-rings. Their paper has paved the way to
understand these important symmetric functions better and also contained
some determinant formulae for the irreducible characters. In our current approach
we will emphasize the role of *-operators which can lead to several generalized
formulae, e.g. Eqs. (\ref{e:frobenius3}, \ref{e:frobenius4}, \ref{e:frobenius5}, \ref{e:frobenius6}).
In \cite{FH} determinant
identities for symplectic and orthogonal symmetric polynomials are discussed from matrix consideration,
so vertex operator or $\lambda$-ring approach can be
viewed as another way to prove these determinant formulae (and more identities)
for infinitely many variables.

\section{Ring of symmetric functions}

\hskip\parindent Let $\Lambda=\Lambda_{\mathbb Q}$ be the ring of symmetric functions
 in countably many variables $x_1,x_2, \cdots$ over the field $\mathbb Q$ of rational numbers.
The degree of homogeneous symmetric functions give rise to a natural gradation for $\Lambda_{\mathbb{Q}}$:
\begin{equation}
\Lambda_{\mathbb{Q}}=\bigoplus_{k\geq 0}\Lambda^k_{\mathbb{Q}},
\end{equation}
where $\Lambda^k_{\mathbb{Q}}$~consists of
   the homogeneous symmetric functions of degree $k$.

We recall some basic notations following \cite{Ma}. A {\it partition} is any sequence $\lambda=(\lambda_1,\lambda_2,\cdots,\lambda_r,\cdots)$ of non-negative integers in decreasing order$\colon$$\lambda_1\geq\lambda_2\geq\cdots\geq\lambda_r\geq\cdots$
with only finitely many non-zero terms.~The non-zero~$\lambda_i$~are called the parts of~$\lambda.$~The number of the parts is the length of~$\lambda,$~denoted by~$l(\lambda);$~and the summation of the parts is the weight of~$\lambda,$~denoted by~$|\lambda|$$\colon$$|\lambda|=\lambda_1+\lambda_2+\cdots$. The partition $\lambda$ can be visualized by
 its Ferrers diagram or Young diagram formed by
 aligning $l$ rows of boxes such that there are exactly $\lambda_i$ boxes on the
 $i$th row. If one reflects the Ferrers diagram  along the main diagonal (the $-45^{\circ}$-axis) , the
 associated partition is called the conjugate $\lambda'$ of $\lambda$. The Frobenius
 notation $\lambda=(\al|\beta)=(\al_1\cdots \al_r|\beta_1\cdots \beta_r)$ of the
 Ferrers diagram describes the partition by
 $\al_i=\lambda_i-i, \beta_i=\lambda_i'-i$, where $r$ is the length of the main diagonal of
 $\lambda$.

 If the parts $\lambda_i$ are
not necessarily in descending order, $\lambda$ is called a {\it composition}
and we also
use the same notation $|\lambda|$ for its weight.
If~$|\lambda|=n,$~we say that $\lambda$ is a partition of~$n$.~We also use the other notation $\colon$$\lambda=(1^{m_1}2^{m_2}\cdots r^{m_r}\cdots)$ to mean that exactly $m_i$ of the parts of $\lambda$ are equal to $i$.~The set of partitions will be denoted by $\mathcal{P}$.

The ring $\Lambda_{\mathbb Q}$ has several families of linear bases indexed by partitions. The well-known ones are
the monomial functions $\{m_{\lambda}\}$, the complete homogeneous symmetric functions $h_{\lambda}$, the elementary symmetric function functions $\{e_{\lambda}\}$ and
the power sum symmetric functions $\{p_{\lambda}\}$. They are respectively determined by their finite counterparts:

(i) $m_{\lambda}(x_1, \cdots, x_n)=x^{\lambda}+\mbox{distinct permutations of $x^{\lambda}$}$;

(ii) $h_{k}(x_1, \cdots, x_n)=\sum_{i_1\leq\cdots\leq i_k}x_{i_1}\cdots x_{i_k}$,
and $h_{\lambda}=h_{\lambda_1}\cdots h_{\lambda_l}$;

(ii)
$e_k(x_1, \cdots, x_n)=\sum_{i_1<\cdots<i_k}x_{i_1}\cdots x_{i_k}$, and $e_{\lambda}=e_{\lambda_1}\cdots e_{\lambda_l}$;

(iii) $p_k(x_1, \cdots, x_n)=\sum\limits_{i=1}^{n} x^k_i$, and $p_{\lambda}=p_{\lambda_1}\cdots p_{\lambda_l}$.

The first three bases are in fact $\mathbb Z$-bases, and the power-sum basis $p_\lambda$ is over $\mathbb{Q}$.
The standard inner product in $\Lambda_{\mathbb{Q}}$ is defined by requiring that the power sum symmetric functions are
orthogonal:
\begin{equation}
<p_\lambda,p_\mu>=z_\lambda\delta_{\lambda\mu},
\end{equation}
where
$z_\lambda=\prod\limits_{i}i^{m_i}m_i!$ for $\lambda=(1^{m_1}2^{m_2}\cdots)$. Under this inner product the Schur symmetric functions
are othornormal and are triangular linear combination of the complete homogeneous symmetric functions. For each partition $\lambda$, the Schur function
is defined by
\begin{equation}
\displaystyle s_{\lambda}(x_1, \cdots, x_n)=\frac{\sum_{\sigma\in S_n} sgn(\sigma) x^{\sigma(\lambda+\delta)}}
{\prod_{i<j}(x_i-x_j)},  
\end{equation}
where $\delta=(n-1, n-2, \cdots, 1, 0)$.

Both $h_n$ and $e_n$ are special Schur functions. In fact,
\begin{equation}
h_n=s_{(n)}, \qquad e_n=s_{(1^n)}.
\end{equation}
Their generating functions are expressed in terms of the power-sum $p_n$:
\begin{align}\label{e:homogeneoussymm}
\sum_{n\geq 0}h_nz^n&=exp(\sum_{n=1}^{\infty}\frac{p_n}nz^n), \\
\sum_{n\geq 0}e_nz^n&=exp(-\sum_{n=1}^{\infty}\frac{p_n}n(-z)^n).  \label{e:elementarysymm}
\end{align}
The Jacobi-Trudi formula \cite{Ma} expresses the Schur functions in terms of $h_n$ or $e_n$:
\begin{equation}
s_{\lambda}=det(h_{\lambda_i-i+j})=det(e_{\lambda_i'-i+j}),
\end{equation}
where $\lambda'$ is the conjugate of $\lambda$.

We also define an involution $\omega$: $\Lambda\to \Lambda$ by $\omega(p_n)=(-1)^{n-1}p_n$.
Then it follows that $\omega(h_n)=e_n$. Subsequently we have
\begin{equation}
\omega(s_{\lambda})=s_{\lambda'}.
\end{equation}

\section{Vertex operators and symmetric functions}

In this section, we define certain vertex
operators to realize the Schur symplectic and orthogonal functions
inside $\Lambda$.
These symmetric functions
are studied by Baker \cite{B} using vertex operators with
middle terms, which are not necessary for our purpose (see \cite{J1, J3}). Our approach is based on \cite{J1, J3}, which enables us to prove the duality between symplectic and orthogonal
Schur functions directly. 

 First we turn the ring $\Lambda_{\mathbb Q}$ into a Fock space for the infinite dimensional Heisenberg
 algebra. In fact let $a_{-n}=p_n$ for $n\geq 1$, the multiplication operator on $\Lambda$, and $a_n=n\frac{\partial}{\partial p_n},$~where $p_n(x)=\sum\limits_{i=1}^{\infty} x^n_i$
is the power sum symmetric function.~Then $\{a_n|n\neq 0\}$ and $c=I$ generate the infinite
 dimensional Heisenberg algebra $\mathcal H$ inside $End(\Lambda)$:
 \begin{equation}
 [a_m,a_n]=m\delta_{m,-n}c, \quad [c,a_n]=0.
 \end{equation}
The space $\Lambda$ is then the unique irreducible representation of $\mathcal{H}$ such that
$a_{n}.1=0$ for $n>0$ and $c=1$. The natural hermitian
structure on $\Lambda$ is given by $a_n^*=a_{-n}$. The monomial basis $a_{-\lambda}=p_{\lambda}$ is orthogonal
and
\begin{equation}
<a_{-\lambda}, a_{-\mu}>=z_{\lambda}\delta_{\lambda\mu}.
\end{equation}

First we recall the vertex operator construction of Schur symmetric functions.
Let $S(z)$ and $S^*(z)$ be the Bernstein vertex operators: $\Lambda\rightarrow \Lambda[[z, z^{-1}]]$ defined by
\begin{align}
S(z)&=exp\big(\sum\limits_{n=1}^{\infty}\frac{a_{-n}}{n}z^n\big)
exp\big(-\sum\limits_{n=1}^{\infty}\frac{a_n}{n}z^{-n}\big)\\
&=exp\big(\sum\limits_{n=1}^{\infty}\frac{p_n}{n}z^n\big)
exp\big(-\sum\limits_{n=1}^{\infty}\frac{\partial}{\partial p_n}z^{-n}\big)
=\sum_{n\in\mathbb Z}S_nz^{-n}, \nonumber
\end{align}
\begin{align}
S^*(z)&=exp\big(-\sum\limits_{n=1}^{\infty}\frac{a_{-n}}{n}z^n\big)
exp\big(\sum\limits_{n=1}^{\infty}\frac{a_n}{n}z^{-n}\big)\\
&=exp\big(-\sum\limits_{n=1}^{\infty}\frac{p_n}{n}z^n\big)
exp\big(\sum\limits_{n=1}^{\infty}\frac{\partial}{\partial p_n}z^{-n}\big)
=\sum_{n\in\mathbb Z}S_n^*z^{n}.\nonumber
\end{align}

In the following result the first realization of
Schur functions by $S(z)$ was given by Bernstein \cite{Z}. The second realization
using the dual operators was first proved in \cite{J1}
and the last realization (\ref{e:frobenius1}) is special case proved in \cite{J3}.
\begin{proposition} \label{P:schur}(\cite{J3}, Th. 2.7) The operator product expansions are given by
\begin{align}
S(z)S(w)&=:S(z)S(w):(1-wz^{-1}),\\
S^*(z)S^*(w)&=:S^*(z)S^*(w):(1-wz^{-1}),\\
S(z)S^*(w)&=:S(z)S^*(w):(1-wz^{-1})^{-1},
\end{align}
where $|w|<min\{|z|, |z|^{-1}\}$, and the last rational functions are understood as power series expansions
in the second variable $w$. Moreover for any partition $\lambda$, the following give four realizations of Schur functions
\begin{align}\label{e:Schur-ip}
s_{\lambda}&=S_{-\lambda_1}\cdots S_{-\lambda_l}.1=(-1)^{|\lambda|}S^*_{\lambda_1'}\cdots S^*_{\lambda_k'}.1,\\
s_{(\al|\beta)}&=(-1)^{|\beta|+r(r-1)/2}S_{-\al_1-1}\cdots S_{-\al_r-r}S_{\beta_1-(r-1)}^*S_{\beta_2-(r-2)}^*\cdots S_{\beta_r}^*.1, \label{e:frobenius1}\\ \label{e:frobenius1}
s_{(\alpha|\beta)}&=(-1)^{|\beta|+r}S_{\beta_1-1}^*\cdots S_{\beta_r-r}^*S_{-\al_1+(r-1)}\cdots S_{-\al_r}.1,
\end{align}
where $(\al|\beta)$ is the Frobenius notation of a partition, i.e., $\al_i\geq \al_{i-1}+1$ and $\beta_i\geq \beta_{i-1}+1$.
\end{proposition}
Note that $S(z).1=exp\big(\sum\limits_{n=1}^{\infty}\frac{a_{-n}}{n}z^n\big)$, therefore
$S_0.1=1$, which means that $S_0$ and similarly $S^*_0$ will not be needed in the realization.

We introduce the vertex operator for the symplectic Schur functions as follows.


\begin{definition} Let $Y(z)$ and $Y^*(z)$ be the vertex operators: $\Lambda\rightarrow \Lambda[[z, z^{-1}]]$ defined by
\begin{align}
Y(z)&=Y(a, z)=exp\big(\sum\limits_{n=1}^{\infty}\frac{a_{-n}}{n}z^n\big)
exp\big(-\sum\limits_{n=1}^{\infty}\frac{a_n}{n}(z^{-n}+z^n)\big)\\
&=exp\big(\sum\limits_{n=1}^{\infty}\frac{p_n}{n}z^n\big)
exp\big(-\sum\limits_{n=1}^{\infty}\frac{\partial}{\partial p_n}(z^{-n}+z^n)\big)
=\sum_{n\in\mathbb Z}Y_nz^{-n}, \nonumber
\end{align}
\begin{align}
Y^*(z)&=(1-z^2)exp\big(-\sum\limits_{n=1}^{\infty}\frac{a_{-n}}{n}z^n\big)
exp\big(\sum\limits_{n=1}^{\infty}\frac{a_n}{n}(z^{-n}+z^n)\big)=(1-z^2)W^*(z)\\
&=(1-z^2)exp\big(-\sum\limits_{n=1}^{\infty}\frac{p_n}{n}z^n\big)
exp\big(\sum\limits_{n=1}^{\infty}\frac{\partial}{\partial p_n}(z^{-n}+z^n)\big)
=\sum_{n\in\mathbb Z}Y_n^*z^{n}, \nonumber
\end{align}
\end{definition}
The operator $Y(z)$ coincides with $V^{(1^2)}(z)$ in \cite{FJK} and was also studied
in \cite{SZ} in $\lambda$-ring language.
We will emphasize the role of $Y^*(z)$ in our approach.
Note that the intermediate vertex operator $W^*(z)=Y(-a, z)$ is obtained from $Y(a, z)$ by
formally changing $a_n$ to $-a_n$. Below we will use $W^*(z)$ for the orthogonal Schur functions.

The operators $Y_n$ and $Y_n^*$ are well-defined operators on the Fock space $\Lambda$.
If one views the variable $z$ as a complex number, then they are Fourier coefficients of the
vertex operators $Y(z)$ and $Y^*(z)$. The operator $Y_n^*$ and the intermediate operator
$W_n^*$ are related by
\begin{align}
Y^*_n&=W^*_n-W^*_{n-2},\\
W^*_n&=Y^*_n+Y_{n-2}^*+Y_{n-4}^*+\cdots,
\end{align}
where the second identity is viewed as a locally finite operator, i.e. it is well-defined on any finite
dimensional subspace of $\Lambda$.

The normal order product is defined as usual. For example,
\begin{align*}
:Y(z)Y(w):&=exp\big(\sum\limits_{n=1}^{\infty}\frac{p_n}{n}
z^n+\sum\limits_{n=1}^{\infty}\frac{p_n}{n}w^n\big)\\[0.4cm]
&\cdot exp\big(-\sum\limits_{n=1}^{\infty} n\frac{\partial}{\partial p_n}
(\frac{z^n+z^{-n}}{n}+\frac{w^n+w^{-n}}{n})\big). 
\end{align*}

\begin{proposition} The operator product expansions are given by
\begin{align}\label{ope1}
Y(z)Y(w)&=:Y(z)Y(w):(1-wz^{-1})(1-wz),\\
Y^*(z)Y^*(w)&=:W^*(z)W^*(w):(1-z^2)(1-w^2)(1-wz^{-1})(1-wz),\\
Y(z)Y^*(w)&=:Y(z)W^*(w):\frac{1-w^2}{(1-wz^{-1})(1-wz)},\label{ope2}\\
Y^*(z)Y(w)&=:Y^*(z)W(w):\frac{1-z^2}{(1-wz^{-1})(1-wz)}, \label{ope3}
\end{align}
where $|w|<min\{|z|, |z|^{-1}\}$, and the rational functions appeared on the right are understood as power series expansions in the second variable $w$.
\end{proposition}

\begin{theorem} \label{T:comm1} The operators $Y_n$ and $Y_n^*$ satisfy the following generalized
Clifford algebra relations:
\begin{align*}
Y_mY_n+Y_{n+1}Y_{m-1}&=0, \\ 
Y^*_mY^*_n+Y^*_{n-1}Y^*_{m+1}&=0,\\
Y_mY_n^*+Y_{n+1}^*Y_{m+1}&=\delta_{m, n}.
\end{align*}
\end{theorem}
\begin{proof}
It follows from (\ref{ope1}) that
\begin{align*}
zY(z)&Y(w)+wY(w)Y(z)\\
&=:Y(z)Y(w):\{(z-w)(1-zw)+(w-z)(1-wz)\}=0.
\end{align*}
Taking coefficients one proves the first two relations. Similarly  we derive the following relation by using (\ref{ope2}-\ref{ope3}).
\begin{align*}
&z^{-1}Y(z)Y^*(w)+w^{-1}Y^*(w)Y(z)\\
&=:Y(z)W^*(w):\frac{1-w^2}{1-zw}\{(z-w)^{-1}+(w-z)^{-1}\}\\
&=:Y(z)W^*(w):\frac{1-w^2}{1-zw} z^{-1}\delta(\frac wz)\\
&=z^{-1}\delta(\frac wz)
\end{align*}

\end{proof}

Besides the operator $W^*(z)$, we also introduce the following vertex operators for the orthogonal Schur functions.

\begin{definition} The vertex operator $W(z)$
is defined as the vertex operators from $\Lambda$ to $\Lambda[[z, z^{-1}]]$  given by
\begin{align}
W(z)&=(1-z^2)exp\big(\sum\limits_{n=1}^{\infty}\frac{p_n}{n}z^n\big)
exp\big(-\sum\limits_{n=1}^{\infty}\frac{\partial}{\partial p_n}(z^{-n}+z^n)\big)=(1-z^2)Y(z)\\
&=\sum\limits_{n\in\mathbb{Z}}W_nz^{-n}.   \nonumber
\end{align}
\end{definition}
We remark that the operator $W(z)$ was denoted as $V^{(2)}(z)$ in \cite{FJK}.
\begin{proposition} \label{P:comm2} The operators $W_n$ and $W_n^*$ satisfy the following generalized
Clifford algebra relations:
\begin{align*}
W_mW_n+W_{n+1}W_{m-1}&=0, \\
W^*_mW^*_n+W^*_{n-1}W^*_{m+1}&=0,\\
W_mW_n^*+W_{n+1}^*W_{m+1}&=\delta_{m,n}.
\end{align*}
Moreover, the orthogonal and symplectic vertex operators are related by:
\begin{align*}
W_n&=Y_n-Y_{n-2}, \qquad Y^*_n=W^*_n-W^*_{n+2},\\
Y_n&=W_n+W_{n-2}+\cdots, \qquad W^*_n=Y^*_n+Y^*_{n+2}+\cdots.
\end{align*}
\end{proposition}

In other words, the operators $W_n$ and $W_n^*$ satisfy the same relations as $Y_n$ and $Y_n^*$
do in Theorem \ref{T:comm1}. The following result explains why we introduce $Y^*(z)$ even though
it is not the dual operator of $Y(z)$.

\begin{theorem} \label{T:dual-bases}
For any partition $\lambda=(\lambda_1, \cdots, \lambda_l)$ and its conjugate partition $\lambda'=(\lambda_1', \cdots, \lambda_k')$
one has that
\begin{align}
Y_{-\lambda_1}\cdots Y_{-\lambda_l}.1&=(-1)^{|\lambda|}Y_{\lambda_1'}^*\cdots Y_{\lambda_k'}^*.1,\\
W_{-\lambda_1}\cdots W_{-\lambda_l}.1&=(-1)^{|\lambda|}W_{\lambda_1'}^*\cdots W_{\lambda_k'}^*.1.
\end{align}
Both are $\mathbb Z$-bases in $\Lambda$.
\end{theorem}
\begin{proof} The two identities are proved exactly the same, so we only treat the first one. First of all, it is clear that
$\{Y_{-\lambda_1}\cdots Y_{-\lambda_l}.1\}$ and $\{Y^*_{\lambda_1'}\cdots Y^*_{\lambda_k'}.1\}$ are bases of $\Lambda$.
We now show that the respective inner products of $Y(z)'s$ and $Y^*(z)'s$ with the Schur basis constructed in
Proposition \ref{P:schur} are equal. To simplify computation we use the exponential
operator $exp(\sum_{n=1}^{\infty}\frac{a_{-n}^2-a_{-2n}}{2n}w^n)$ to relate
$Y(z)$ and $S(z)$. In fact, for any two vectors $u, v\in\Lambda$ we have
\begin{align*}
&<Y^*(z)u, exp(\sum_{n=1}^{\infty}\frac{a_{-n}^2-a_{-2n}}{2n}w^n)v>\\
&=<exp(\sum_{n=1}^{\infty}\frac{a_{n}^2-a_{2n}}{2n}w^n)Y^*(z)u, v>\\
&=<(1-wz^2)^{-1}Y^*(z)exp(-\sum_{n=1}^{\infty}\frac{a_n}n(zw)^n)exp(\sum_{n=1}^{\infty}\frac{a_{n}^2-a_{2n}}{2n}w^n)u, v>
\end{align*}
Notice that $\displaystyle \lim_{w\to 1}(1-wz^2)^{-1}Y^*(z)exp(-\sum_{n=1}^{\infty}\frac{a_n}n(zw)^n)=S^*(z)$.
Subsequently one has for
any $k, l$
\begin{align}\label{e:ip1}
&<Y^*(z_1)\cdots Y^*(z_l).1, exp(\sum_{n=1}^{\infty}\frac{a_{-n}^2-a_{-2n}}{2n})S(w_1)\cdots S(w_k).1>\\ \nonumber
&=<S^*(z_1)\cdots S^*(z_l).1, S(w_1)\cdots S(w_k).1>
\end{align}
Similar computation also gives that
\begin{align}\label{e:ip2}
&<Y(z_1)\cdots Y(z_l).1, exp(\sum_{n=1}^{\infty}\frac{a_{-n}^2-a_{-2n}}{2n})S(w_1)\cdots S(w_k).1>\\ \nonumber
&=<S(z_1)\cdots S(z_l).1, S(w_1)\cdots S(w_k).1>
\end{align}
By comparing coefficients of $z^{\lambda}w^{\nu}$ in Eqs. (\ref{e:ip1}-\ref{e:ip2}) and
using Eq. (\ref{e:Schur-ip}) it follows that for any two partitions $\lambda$ and $\nu$
\begin{align*}
&<Y_{-\lambda_1}\cdots Y_{-\lambda_l}.1, S_{-\nu_1}\cdots S_{-\nu_m}.1>\\
&=(-1)^{|\lambda|}<Y^*_{\lambda_1'}\cdots Y^*_{\lambda_k'}.1, S_{-\nu_1}\cdots S_{-\nu_m}.1>,
\end{align*}
Therefore $Y_{-\lambda_1}\cdots Y_{-\lambda_l}.1=(-1)^{|\lambda|}Y^*_{\lambda_1'}\cdots Y^*_{\lambda_k'}.1$. The proof also shows that the vectors
$\{Y_{-\lambda_1}\cdots Y_{-\lambda_l}.1\}$ form a $\mathbb Z$-basis of $\Lambda$.
\end{proof}

\section{Determinant formulae and duality}

For any partition $\lambda=(\lambda_1,\lambda_2,\cdots,\lambda_k)$
we define the symplectic and orthogonal Schur functions (cf. \cite{KT}) respectively as follows.
\begin{align}\label{e:symp}
 sp_\lambda &=\frac12 det(h_{\lambda_i-i+j}+h_{\lambda_i-i-j+2})_{1\leq i,j\leq k}, \\
 o_{\lambda}&=det(h_{\lambda_i-i+j}-h_{\lambda_i-i-j})_{1\leq i,j\leq k}. \label{e:ortho}
\end{align}
 where $h_m$ is the $m$th complete symmetric function in $\Lambda$. Note that the factor
 $\frac 12$ does not affect that both elements are in $\Lambda_{\mathbb Z}$.
 In fact,
 \begin{equation}
 sp_{\lambda}=det\begin{bmatrix} h_{\lambda_1} & h_{\lambda_1+1}+h_{\lambda_1-1} &
 \cdots & h_{\lambda_k+k-1}+h_{\lambda_1-k+1}\\
 h_{\lambda_2-1} & h_{\lambda_2}+h_{\lambda_2-2} &  \cdots & h_{\lambda_2+k-2}+h_{\lambda_2-k}\\
 \vdots & \vdots & \ddots & \vdots\\
 h_{\lambda_k-k+1} & h_{\lambda_k-k+2}+h_{\lambda_k-k} &\cdots & h_{\lambda_k}+h_{\lambda_k-2k+2}
 \end{bmatrix}.
 \end{equation}

The orthogonal/symplectic Schur functions come from Weyl characters for classical Lie groups. In fact
$sp_\lambda(x_1, x_1^{-1}, \cdots, x_n, x_n^{-1})=\frac12det(h_{\lambda_i-i+j}-h_{\lambda_i-i-j+2})$ is the character of the irreducible
$Sp_{2n}$-module with weight $\lambda$. Here the complete
homogeneous function $h_m$ is defined by
$\prod\limits_{i=1}^n\frac1{(1+x_iz)(1+x^{-1}_iz)}=\sum\limits_{m=0}^{\infty} h_mz^m.$
In particular if $V$ is the defining module of $Sp_{2n}$, then $S^n(V)$ is irreducible, but $\Lambda^n(V)+\Lambda^{n-2}(V)+\cdots$ is the fundamental irreducible representation.

Similarly for $SO(2n+1)$ the character
$ch(V(\lambda)$ is also given by the orthogonal Schur function, and the irreducible
character of $SO_{2n}$-module $Res_{SO(2n)}^{O(2n)}V(\lambda)$ associated with
the highest weight $\lambda=\lambda_1\epsilon_1+\cdots +\lambda_n\epsilon_n$ is
given by $
o_{\lambda}(t_1, t_1^{-1}\cdots, t_n, t_n^{-1})=\chi_{\lambda}+\chi_{\sigma(\lambda)}.
$

In both cases the Schur orthogonal symmetric function is defined by
$o_{\lambda}(x_1 \cdots, x_n)=det(e_{\lambda_i-i-j}-e_{\lambda_i-i+j})$ where the elementary symmetric functions $e_m$ are given by $
(1+z)\prod\limits_{i=1}^n{(1+x_iz)(1+x^{-1}_iz)}=\sum\limits_{m=0}^{\infty} e_mz^m$
or $\prod\limits_{i=1}^n{(1+x_iz)(1+x^{-1}_iz)}=\sum\limits_{m=0}^{\infty} e_mz^m,
$
respectively for type $B$ and $D$.

Before giving vertex operator realization of $sp_{\lambda}$ and $o_{\lambda}$,
we need several Vandermonde type identities.

\begin{lemma} \label{L:van}
For any positive integer $k$, the following Vandermonde-like identities hold.
\begin{align}\label{e:Van1}
det&(z_i^{k-j}+z_i^{k+j-2})=2\prod_{1\leq i<j\leq k}(z_i-z_j)(1-z_iz_j)\\ \nonumber
&=\sum_{\sigma\in S_k, \epsilon_i=\pm 1}sgn(\sigma)(z_1\cdots z_k)^{k-1}z_1^{\epsilon_1(\sigma(1)-1)}\cdots z_k^{\epsilon_k(\sigma(k)-1)}\\ \label{e:Van2}
det&(z_i^{k-j}-z_i^{k+j})=\prod_{1\leq i<j\leq k}(z_i-z_j)\prod_{1\leq i\leq j\leq k}(1-z_iz_j)\\ \nonumber
&=\sum_{\sigma\in S_k, \epsilon_i=\pm 1} sgn(\sigma)\epsilon_1\cdots \epsilon_k z_1^{k-\epsilon_1\sigma(1)}\cdots z_k^{k-\epsilon_k\sigma(k)}\\ \label{e:Van3}
det&(z_i^{j-1}-z_i^{2k-j+1})=\prod_{1\leq i<j\leq k}(z_j-z_i)\prod_{1\leq i\leq j\leq k}(1-z_iz_j)\\ \nonumber
&=\sum_{\sigma\in S_k, \epsilon_i=\pm 1} sgn(\sigma)\epsilon_1\cdots \epsilon_k
(z_1\cdots z_k)^kz_1^{\epsilon_1(-k+\sigma(1)-1)}\cdots z_k^{\epsilon_k(-k+\sigma(k)-1)}.
\end{align}
\end{lemma}
\begin{proof} These formulae are special cases of Weyl's denominator formulae
\cite{W}. We include a proof for completeness.
The Weyl denominator formula of type $D$ \cite{W} says that (in reversing order of columns):
\begin{equation}\label{e:denom-B}
\sum_{\sigma\in S_k, \epsilon_i=\pm 1}sgn(\sigma)z_1^{\epsilon_1(\sigma(1)-1)}\cdots z_k^{\epsilon_k(\sigma(k)-1)}
=2(z_1\cdots z_k)^{1-k}\prod_{1\leq i<j\leq k}(z_i-z_j)(1-z_iz_j).
\end{equation}

By the anti-symmetry of the summation side, one sees that
the left-hand side is equal to $det(z_i^{j-1}+z_i^{-j+1})$. Note that
\begin{align*}
det&(z_i^{j-1}+z_i^{-j+1})=2|1, z+z^{-1}, \cdots, z^{k-1}+z^{1-k}|\\
&=2z_2^{-1}z_3^{-2}\cdots z_k^{1-k}|1, 1+z^2, 1+z^4, \cdots, 1+z^{2k-2}|\\
&=2(z_1\cdots z_k)^{1-k}|z^{k-1}, z^{k-2}+z^{k}, \cdots, 1+z^{2k-2}|\\
&=(z_1\cdots z_k)^{1-k}det(z_i^{k-j}+z_i^{k+j-2}),
\end{align*}
where we have displayed a typical row in the determinant computation.

Similarly one form of Weyl denominator formula of type $C$ \cite{W} says that
\begin{align*}
det&(z_i^{-j}-z_i^{j})=(z_1\cdots z_k)^{-k}\prod_{1\leq i<j\leq k}(z_i-z_j)\prod_{1\leq i\leq j\leq k}(1-z_iz_j)\\
&=\sum_{\sigma\in S_k, \epsilon_i=\pm 1} sgn(\sigma)\epsilon_1\cdots \epsilon_k z_1^{-\epsilon_1\sigma(1)}\cdots z_k^{-\epsilon_k\sigma(k)}.
\end{align*}
The left-hand side can be easily changed to our current form:
\begin{align*}
det&(z_i^{-j}-z_i^{j})=|z^{-1}-z, z^{-2}-z^2, \cdots, z^{-k}-z^{k}|\\
&=(z_1\cdots z_k)^{-k}|z^{k-1}-z^{k+1}, z^{k-2}-z^{k+2}, \cdots, 1-z^{2k}|\\
&=det(z_i^{k-j}-z^{k+j}_i).
\end{align*}
Finally the last identity is obtained from Eq. (\ref{e:Van2}) by reversing the order of columns.
\end{proof}

We remark that one can also prove Lemma \ref{L:van} exclusively based on the
Vandermonde identities, which will make our later derivation of Weyl formulae independent from
Weyl denominator formulae.
Now we are ready for vertex operator realization of symplectic Schur functions.

\begin{theorem} \label{T:symp} For any partition $\lambda=(\lambda_1, \cdots, \lambda_k)$ we have
\begin{equation}\label{e:det-symp1}
Y_{-\lambda_1}Y_{-\lambda_2}\cdots Y_{-\lambda_k}.1=sp_{\lambda}=
\frac 12det(h_{\lambda_i-i+j}+h_{\lambda_i-i-j+2}).
\end{equation}
\end{theorem}
\begin{proof} First of all, we recall Eq. (\ref{e:homogeneoussymm})
\begin{equation*}
H(z)=exp\big(\sum\limits_{n=1}^{\infty} a_{-n}\frac{z^n}{n}\big)=\sum\limits_{n=0}^{\infty}
h_nz^n.   
\end{equation*}
where $h_n$ is the complete symmetric function.

By the standard normal order product and Wick's theorem \cite{FLM} it follows that
\begin{equation}
Y(z_1)Y(z_2)\cdots Y(z_k).1=\prod_{i<j}(1-\frac{z_j}{z_i})(1-z_iz_j) H(z_1)\cdots H(z_k),
\end{equation}
where $1>|z_1|>\cdots >|z_k|$.

Let $C_1$, $C_2, \cdots, C_k$ be concentric circles with decreasing radii $<1$, and take $z_1, \cdots, z_k$ as complex numbers. Then
\begin{align*}
Y_{-\lambda_1}&Y_{-\lambda_2}\cdots Y_{-\lambda_k}.1\\
=&\frac{1}{(2\pi i)^k}\int_{C_1}\cdots \int_{C_k}\prod\limits_{1\leq i<j\leq
k}(1-\frac{z_j}{z_i})(1-z_iz_j)H(z_1)\cdots H(z_k)\frac{dz}{z^{\lambda+\iota}},
\end{align*}
where we denote $z^{\lambda}=z_1^{\lambda_1}\cdots z_k^{\lambda_k}$ for any composition
$\lambda=(\lambda_1, \cdots, \lambda_k)$, $\iota=(1, \cdots, 1)$ is of length $k$, and $dz=dz_1\cdots z_k$.

Recall the Vandermonde type identity (\ref{e:Van1}):
\begin{align*}
\frac{1}{2}det(z^{k-j}_i+z^{k+j-2}_i)=&\prod\limits_{1\leq i< j\leq k}(1-z_iz_j)
\prod\limits_{1\leq i<j\leq k}(z_i-z_j)\\
=&(z_1\cdots z_k)^{k-1}\frac 12\sum\limits_{\sigma\in S_k,\epsilon_i=\pm 1}
sgn(\sigma)z^{\epsilon_1(\sigma(1)-1)}_1\cdots z_k^{
\epsilon_k(\sigma(k)-1)}.
\end{align*}

We obtain that $Y_{-\lambda_1}Y_{-\lambda_2}\cdots Y_{-\lambda_k}.1$ is equal to
\begin{align*}
&\frac{1}{(2\pi i)^k}\int_{C_1\cdots C_k}\prod_{i=1}^k exp\big(\sum\limits_{n=1}^{\infty}
\frac{a_{-n}}{n}z_i^n\big)\prod\limits_{1\leq i<j\leq k}(1-\frac{z_j}{z_i})(1-z_iz_j)
\frac{dz}{z^{\lambda+\iota}}\\
=&\sum_{n_1\geq 0\cdots n_k\geq 0}\frac{1}{(2\pi i)^k}\int_{C_1\cdots C_k}h_{n_1}\cdots h_{n_k}z_1^{n_1}\cdots z_k^{n_k}\prod\limits_{1\leq i<j\leq k}(z_i-z_j)(1-z_iz_j)\frac{dz_k\cdots dz_1}{z_1^{\lambda_1+k}\cdots z_k^{\lambda_k+1}}\\
=&\sum_{n_1\geq 0\cdots n_k\geq 0}\frac{1}{(2\pi i)^k}\int_{C_1\cdots C_k}h_{n_1}\cdots h_{n_k}\frac{dz_k\cdots dz_1}{z_k\cdots z_1}\\
&\quad \frac12\sum\limits_{\sigma\in S_k,\epsilon_1,\cdots,\epsilon_k=\pm 1}
sgn(\sigma)z^{\epsilon_1(\sigma(1)-1)-\lambda_1+n_1}_1z^{\epsilon_2(\sigma(2)-1)-\lambda_2+1+n_2}_2\cdots z^{
\epsilon_k(\sigma(k)-1)-\lambda_k+k-1+n_k}_k\\
=&\frac12\sum\limits_{\sigma\in S_k,\epsilon_i=\pm 1}\epsilon_1\cdots\epsilon_ksgn(\sigma)h_{\lambda_1+\epsilon_1(\sigma(1)-1)}
h_{\lambda_2-1+\epsilon_2(\sigma(2)-1)}\cdots h_{\lambda_k-k+1+\epsilon_k(\sigma(k)-1)}\\
=&\frac 12det(h_{\lambda_i-i+j}+h_{\lambda_i-i-j+2})
\end{align*}
\end{proof}

Similarly for the vertex operator $Y^*(z)$ we have the following determinant expression.
We also remark that the formulae involving Frobenius notation seem to be new, see
Eqs. (\ref{e:frobenius3}), (\ref{e:frobenius4}), (\ref{e:frobenius5}) and (\ref{e:frobenius6}).

\begin{theorem} \label{T:symp2} For partition $\lambda=(\lambda_1, \cdots, \lambda_k)$ we have
\begin{equation}\label{e:det-symp2}
Y_{\lambda_1}^*Y_{\lambda_2}^*\cdots Y_{\lambda_k}^*.1= (-1)^{|\lambda |}det(e_{\lambda_i-i+j}-e_{\lambda_i-i-j}).
\end{equation}
Therefore for any partition $\lambda$ and its conjugate $\lambda'$ one has that
\begin{equation}\label{e:det-symp3}
sp_{\lambda}=\frac12det(h_{\lambda_i-i+j}+h_{\lambda_i-i-j+2})=det(e_{\lambda_i'-i+j}-e_{\lambda_i'-i-j}).
\end{equation}
Moreover one has for any partition $(\al|\beta)$ in Frobenius notation
\begin{align}\label{e:frobenius3}
sp_{(\al|\beta)}&=(-1)^{|\beta|+r(r-1)/2}Y_{-\al_1-1}\cdots Y_{-\al_r-r}Y_{\beta_1-(r-1)}^*Y_{\beta_2-(r-2)}^*\cdots Y_{\beta_r}^*.1,\\ \label{e:frobenius4}
sp_{(\al|\beta)}&=(-1)^{|\beta|+r}Y_{\beta_1-1}^*\cdots Y_{\beta_r-r}^*Y_{-\al_1+(r-1)}\cdots Y_{-\al_r}.1.
\end{align}
\end{theorem}
\begin{proof} Let $C_1$, $C_2, \cdots, C_k$ be as above, then invoking (\ref{e:Van2}) we get that
\begin{align*}
&Y^*_{\lambda_1}Y^*_{\lambda_2}\cdots Y^*_{\lambda_k}.1\\
=&\frac{1}{(2\pi i)^k}\int_{C_k\cdots C_1}\prod_{i=1}^k exp\big(-\sum\limits_{n=1}^{\infty}
\frac{a_{-n}}{n}z_i^n\big)\prod\limits_{1\leq i<j\leq k}(1-\frac{z_j}{z_i})\prod\limits_{1\leq i\leq j\leq k}(1-z_iz_j)
\frac{dz}{z^{\lambda+\bold 1}}\\
=&\sum_{n_1\geq 0\cdots n_k\geq 0}\frac{(-1)^{|\bold n|}}{(2\pi i)^k}\int_{C_k\cdots C_1}e_{n_1}\cdots e_{n_k}z^{\bold n}\prod\limits_{1\leq i<j\leq k}(z_i-z_j)\prod_{i\leq j}(1-z_iz_j)\frac{dz}{z^{\lambda+\bold 1}}\\
=&\sum_{n_1\geq 0\cdots n_k\geq 0}\frac{(-1)^{|\bold n|}}{(2\pi i)^k}\int_{C_k\cdots C_1}e_{n_1}\cdots e_{n_k}\frac{dz_k\cdots dz_1}{z_k\cdots z_1}\\
&\qquad\sum\limits_{\sigma\in S_k,\epsilon_1,\cdots,\epsilon_k=\pm 1}
\epsilon_1\cdots\epsilon_ksgn(\sigma)z^{1-\epsilon_1\sigma(1)-\lambda_1+n_1}_1\cdots z^{k-
\epsilon_k\sigma(k)-\lambda_k+n_k}_k\\
=&\sum\limits_{\sigma\in S_k,\epsilon_1,\cdots,\epsilon_k=\pm 1}(-1)^{|\lambda |}\epsilon_1\cdots\epsilon_ksgn(\sigma)e_{\lambda_1+\epsilon_1\sigma(1)-1}e_{\lambda_2+\epsilon_2\sigma(2)-2}\cdots e_{\lambda_k+\epsilon_k\sigma(k)-k}\\
=&(-1)^{|\lambda |}det(e_{\lambda_i-i+j}-e_{\lambda_i-i-j}),
\end{align*}
where $z^{\bold n}=z_1^{n_1}\cdots z_k^{n_k}$. This completes the proof of (\ref{e:det-symp2}), and the identity (\ref{e:det-symp3}) then follows as a consequence of Theorem \ref{T:dual-bases}.

The formulae (\ref{e:frobenius3}) and (\ref{e:frobenius3}) in Frobenius notation are consequences of
Theorem \ref{T:dual-bases} and Theorem \ref{T:symp}.
In fact (\ref{e:frobenius3}) is derived by first rewriting $Y_{\beta_1-(r-1)}^*Y_{\beta_2-(r-2)}^*\cdots Y_{\beta_r}^*.1$
back to $(-1)^{|\mu|}Y_{\mu}.1$ using Theorem \ref{T:dual-bases} and then reapply Theorem \ref{T:symp},
where $\mu$ is the conjugate diagram of $(\beta_1-(r-1), \beta_2-(r-2), \cdots, \beta_r)$. The
second formula (\ref{e:frobenius4}) is proved similarly.
\end{proof}

We remark that Eq. (\ref{e:det-symp2}) can also be proved by using the third Vandermonde-like identity. In fact we can rewrite the identity as
\begin{equation}
Y^*_{\lambda_1}Y^*_{\lambda_2}\cdots Y^*_{\lambda_k}.1=(-1)^{|\lambda|+k(k-1)/2}det(e_{\lambda_i-i-j+k+1}-e_{\lambda_i-i+j-k-1}).
\end{equation}
This is the same determinant (\ref{e:det-symp2}) by reversing the columns.

For the orthogonal case we have the following result for the vertex operators $W(z)$ and $W^*(z)$,
which is shown in the same way as Theorem \ref{T:symp}--\ref{T:symp2}.
\begin{theorem} \label{T:ortho} For partition $\lambda=(\lambda_1, \cdots, \lambda_k)$ we have
\begin{align}\label{e:det-ortho1}
W_{-\lambda_1}W_{-\lambda_2}\cdots W_{-\lambda_k}.1&=o_{\lambda}=
det(h_{\lambda_i-i+j}-h_{\lambda_i-i-j}), \\ \label{e:det-ortho2}
W_{\lambda_1}^*W_{\lambda_2}^*\cdots W_{\lambda_k}^*.1
&= \frac{(-1)^{|\lambda |}}2det(e_{\lambda_i-i+j}+e_{\lambda_i-i-j+2}).
\end{align}
Therefore for any partition $\lambda$ and its conjugate $\lambda'$ one has that
\begin{equation}
o_{\lambda}=det(h_{\lambda_i-i+j}-h_{\lambda_i-i-j})=\frac{1}2det(e_{\lambda'_i-i+j}+e_{\lambda'_i-i-j+2}).
\end{equation}
Moreover for any partition $(\al|\beta)$ in Frobenius notation one has that
\begin{align}
o_{(\al|\beta)}&=(-1)^{|\beta|+r(r-1)/2}W_{-\al_1-1}\cdots W_{-\al_r-r}W_{\beta_1-(r-1)}^*W_{\beta_2-(r-2)}^*\cdots W_{\beta_r}^*.1, \label{e:frobenius5}\\ \label{e:frobenius6}
o_{(\al|\beta)}&=(-1)^{|\beta|+r}W_{\beta_1-1}^*\cdots W_{\beta_r-r}^*W_{-\al_1+(r-1)}\cdots W_{-\al_r}.1.
\end{align}
\end{theorem}

We can now prove the duality between symplectic Schur functions and orthogonal Schur functions.
\begin{theorem} Under the involution $\omega$
\begin{equation}
\omega(sp_{\lambda})=o_{\lambda'}.
\end{equation}
\end{theorem}
\begin{proof} Let $sp_{\lambda}=\sum_{\mu}d_{\lambda\mu}s_{\mu}$. Then it follows from
the matrix coefficients of vertex operators:
\begin{align*}
d_{\lambda\mu}&=<Y_{-\lambda}.1, S_{-\mu}.1>=\mbox{coeff. of $z^{\lambda}w^{\mu}$ of}<Y(z_1)\cdots Y(z_k).1, S(w_1)\cdots S(w_l).1> \\
&=\mbox{CT}_{z^{\lambda}w^{\mu}}\prod_{1\leq i<j\leq l}(1-\frac{w_j}{w_i})\prod_{1\leq i<j\leq k}(1-\frac{z_j}{z_i})(1-z_iz_j)\\
&\qquad\qquad\qquad <:Y(z_1)\cdots Y(z_k):.1, :S(w_1)\cdots S(w_l):.1>\\
&=\mbox{CT}_{z^{\lambda}w^{\mu}}\prod_{1\leq i<j\leq l}(1-\frac{w_j}{w_i})\prod_{1\leq i<j\leq k}(1-\frac{z_j}{z_i})(1-z_iz_j)
\prod_{1\leq i\leq l,1\leq j\leq k}(1-w_iz_j)^{-1},
\end{align*}
where the last identity uses the fact that $:Y(z_1)\cdots Y(z_k):.1=:S(z_1)\cdots S(z_k):.1$
and Proposition \ref{P:schur}.
Similarly let $o_{\lambda}=\sum_{\mu}d_{\lambda\mu}'s_{\mu}$, then using Theorem \ref{T:dual-bases} we have
\begin{align*}
d_{\lambda\mu}'&=<W_{\lambda'}^*.1, S_{\mu'}^*.1>=\mbox{CT}_{z^{\lambda'}w^{\mu'}}<W^*(z_1)\cdots W^*(z_k).1, S^*(w_1)\cdots S^*(w_l).1> \\
&=\mbox{CT}_{z^{\lambda'}w^{\mu'}}\prod_{1\leq i<j\leq l}(1-\frac{w_j}{w_i})\prod_{1\leq i<j\leq k}(1-\frac{z_j}{z_i})(1-z_iz_j)\\
&\qquad\qquad\qquad <:W^*(z_1)\cdots W^*(z_k):.1, :S^*(w_1)\cdots S^*(w_l):.1>\\
&=\mbox{CT}_{z^{\lambda'}w^{\mu'}}\prod_{1\leq i<j\leq l}(1-\frac{w_i}{w_j})\prod_{1\leq i<j\leq k}(1-\frac{z_j}{z_i})(1-z_iz_j)
\prod_{1\leq i\leq l,1\leq j\leq k}(1-w_iz_j)^{-1},
\end{align*}
where we note that $:W^*(z_1)\cdots W^*(z_k):.1=:S^*(z_1)\cdots S^*(z_k):.1$.
Therefore $d_{\lambda\mu}'=d_{\lambda'\mu'}$. Hence $\omega(sp_{\lambda})=o_{\mu'}$ due to $\omega(s_{\lambda})=s_{\lambda'}$.
\end{proof}

If we introduce the symmetric function $\hat{h}_n:=h_n-h_{n-2}$ and $\hat{e}_n:=e_n-e_{n-2}$ for $n\in\mathbb N$, then
\begin{align}
h_n&=\hat{h}_n+\hat{h}_{n-2}+\hat{h}_{n-4}+\cdots\\
e_n&=\hat{e}_n+\hat{e}_{n-2}+\hat{e}_{n-4}+\cdots
\end{align}
Subsequently $\Lambda=\mathbb Z[h_1, h_2, \cdots ]=\mathbb Z[\hat{h}_1, \hat{h}_2, \cdots ]
=\mathbb Z[e_1, e_2, \cdots ]=\mathbb Z[\hat{e}_1, \hat{e}_2, \cdots ]$.
We also introduce $\check{h}_n$ and $\check{e}_n$ by
\begin{align}
\check{h}_n&=h_n+h_{n-2}+h_{n-4}+\cdots\\
\check{e}_n&=e_n+e_{n-2}+e_{n-4}+\cdots
\end{align}

\begin{theorem} \label{T:fourdet}
In terms of the generators $h_n, e_n$ and the new generators
$\hat{h}_n$, $\hat{e}_n$, $\check{h}_n$ and $\check{e}_n$, we have that
\begin{align}\label{det1}
sp_{\lambda}&=\frac12det(h_{\lambda_i-i+j}+h_{\lambda_i-i-j+2})=det(\check{h}_{\lambda_i-i+j}-\check{h}_{\lambda_i-i-j})\\
&=det(e_{\lambda'_i-i+j}-e_{\lambda'_i-i-j})=\frac12det(\hat{e}_{\lambda'_i-i+j}+\hat{e}_{\lambda'_i-i-j+2}),\label{det2}\\
o_{\lambda}&=det({h}_{\lambda_i-i+j}-{h}_{\lambda_i-i-j})=\frac12det(\hat{h}_{\lambda_i-i+j}+\hat{h}_{\lambda_i-i-j+2})
\label{det3}\\
&=\frac12det(e_{\lambda'_i-i+j}+e_{\lambda'_i-i-j+2})=det(\check{e}_{\lambda'_i-i+j}-\check{e}_{\lambda'_i-i-j}). \label{det4}
\end{align}
\end{theorem}
\begin{proof} Consider the determinant $sp_{\lambda}=det(e_{\lambda'_i-i+j}-e_{\lambda'_i-i-j})$.
If we denote for $n\in \mathbb Z_+$ the column vector
\begin{equation*}
[e_n]=\begin{bmatrix} e_n \\ e_{n-1} \\ \vdots \\ e_{n-k+1}\end{bmatrix}
\end{equation*}
Then  $sp_{\lambda}=|[e_{\lambda'_1}]-[e_{\lambda'_1-2}], [e_{\lambda'_1+1}]-[e_{\lambda'_1-3}], \cdots,
[e_{\lambda'_1+k-1}]-[e_{\lambda'_1-k-1}]|$. Note that the $i$th column
\begin{align*}
[e_{\lambda'_1+i-1}]-[e_{\lambda'_1-i-1}]&=[\hat{e}_{\lambda'_1+i-1}]+[\hat{e}_{\lambda'_1+i-3}]+\cdots + [\hat{e}_{\lambda'_1-i+1}]\\
&=[\hat{e}_{\lambda'_1+i-1}]+ [\hat{e}_{\lambda'_1-i+1}]+ ([\hat{e}_{\lambda'_1+i-3}]+\cdots +[\hat{e}_{\lambda'_1-i+3}])
\end{align*}
where the parentheses are exactly the $(i-2)$th column, so it can be removed. Thus by successively subtracting
from previous columns we have that
\begin{equation*}
sp_{\lambda}=|[\hat{e}_{\lambda'_1}], [\hat{e}_{\lambda'_1+1}]+[\hat{e}_{\lambda'_1-1}], \cdots, [\hat{e}_{\lambda'_1+k-1}]+[\hat{e}_{\lambda'_1-k+1}]|,
\end{equation*}
which is exactly $\frac12det(\hat{e}_{\lambda'_i+j-i}+\hat{e}_{\lambda'_i-j-i+2})$. The other identities are proved similarly.
\end{proof}
We remark that the left identities (\ref{det1}) and (\ref{det3}) were due to Weyl \cite{W} and the left
identities (\ref{det2}) and (\ref{det4}) were
found by Koike and Terado \cite{KT} (see also \cite{FH}, \cite{SZ}). The right identities (\ref{det2}) and (\ref{det4})
are consequences of the duality and
the right identities (\ref{det3}) and (\ref{det1}), which were due to
Shimozono-Zabrocki (cf. Prop. 12 in \cite{SZ}).

\medskip

\centerline{\bf Acknowledgments}
We thank the referees for stimulative remarks which have
made the paper better. NJ gratefully acknowledges the support of
Simons Foundation, Humboldt Foundation, and NSFC as well as
Max-Planck Institute for Mathematics, Bonn
and Max-Planck Institute for Mathematics
in the Sciences, Leipzig for hospitality during the work.

\bigskip

\bibliographystyle{amsalpha}

\end{document}